\documentclass[sn-mathphys-ay]{sn-jnl}
\usepackage{floatrow}
\usepackage{anyfontsize}
\usepackage{verbatim}
\usepackage{graphicx}
\usepackage{enumitem}
\usepackage{multirow}
\usepackage{amsmath,amssymb,amsfonts}
\usepackage{amsthm}
\usepackage{mathrsfs}
\usepackage[title]{appendix}
\usepackage{xcolor}
\usepackage{textcomp}
\usepackage{manyfoot}
\usepackage{booktabs}
\usepackage{algorithm}
\usepackage{algorithmicx}
\usepackage{algpseudocode}
\usepackage{listings}
\usepackage{graphics,subfigure}
\newtheorem{theorem}{Theorem}

\newtheorem{definition}{Definition}

\begin{document}

\title[Generalized Finite Difference Method for Solving Stochastic
Diffusion Equations]{Generalized Finite Difference Method for Solving Stochastic
Diffusion Equations}

\author{Faezeh Nassajian Mojarrad}
\affil{Max Planck Institute for Informatics, Germany}

\abstract{
Stochastic diffusion equations are crucial for modeling a range of physical phenomena influenced by uncertainties. We introduce the generalized finite difference method for solving these equations. Then, we examine its  consistency, stability and convergence in  mean-square, showing that   the proposed method preserves stability and demonstrates favorable convergence characteristics under suitable assumptions.
In order to validate the methodology, we present numerical results in one-, two-, and three-dimensional space domains.
}


\maketitle

\section{Introduction}
Dynamical systems in science and engineering are most commonly modeled using ordinary differential equations (ODEs).
Real-world phenomena often involve inherent uncertainties.
We can describe these uncertainties using random processes, which allows the incorporation of randomness into ODE models, leading to the concept of stochastic differential equations (SDEs).
Several numerical methods have been developed for solving SDEs, see e.g.
\cite{Kloeden1992,Buckwar2007,Shahmoradi2021,Li2023,sde2024}.

SDEs describe the evolution of random processes over time, whereas stochastic partial differential equations (SPDEs) expand this framework to include spatial dimensions.
SPDEs have been extensively employed to model various applications in mathematical sciences, physics, biology and engineering.
There exists an enormous literature on
solving SPDEs, see e.g. \cite{Cui2017,Anton2020,Chen2020,Jentzen2020,Yang2021,Zhao2020,Nassajian2024,Jentzen2011}.

One of the oldest and simplest methods used for solving flow and heat transfer problems, is the finite difference method \cite{Courant1928}, particularly when the  domain is not complex.
However, this classic method remained limited by the necessity of using regular meshes.
In \cite{Liszka1980} has been developed a finite difference on irregular regions. Then, in \cite{Benito2001} an explicit form of the meshless finite difference method,
 called the generalized finite difference method, was presented. The generalized finite difference method  derived from the finite difference method, enables its application to scattered points and irregular areas.
This method is based on Taylor expansion and the moving least squares method.
In this method, partial derivatives are approximated using a combination of function values and weighting coefficients at each nodal point and its neighbors.
Due to the characteristics of meshless methods, this approach is highly flexible for solving partial differential equations and can be easily applied to irregular domains.
That is why in recent years it has been utilized to address various numerical problems across different scientific disciplines, for example, \cite{Fan2014,Li2017,Fu2019,Li2021,Izadian2013}.

The remainder of this paper is organized as follows. Firstly, the stochastic generalized finite difference scheme is introduced
in Section \ref{sec:methodology}.
In Section \ref{sec:stability}, we
analyze the consistency, stability and convergence of the proposed method.
Numerical examples  are presented in Section \ref{sec:numerics}. In Section \ref{sec:conclusion}, we  discuss our results, and outline potential directions for future research.

\section{Stochastic Generalized Finite Difference scheme}
\label{sec:methodology}
We are interested in the solution of the stochastic diffusion equation
\begin{subequations}
\label{eq:main}
\begin{align}
v_t(\mathbf{x} ,t)&=  \rho \nabla^2 v(\mathbf{x} ,t)+\mu v(\mathbf{x},t)dW(t),~ \mathbf{x}\in \Omega \subset \mathbb{R}^D,~ t \in \mathbb{R}^+,\label{eq:main1}\\
v(\mathbf{x},0)&=h(\mathbf{x}), ~\mathbf{x}\in \Omega,\\
v(\mathbf{x},t)&=F(\mathbf{x},t),~ \mathbf{x}\in \partial \Omega,~ t \in \mathbb{R}^+.
\end{align}
\end{subequations}
where $\Omega \subset \mathbb{R}^D$ is the spatial domain with the boundary $\partial \Omega$,
$v:\Omega \times \mathbb{R}^+ \to \mathbb{R}$ are the unknowns of the system, $\nabla$ is the gradient operator, $\rho$ represents the viscosity coefficient, $\mu$ is the additive
noise and $W$ denotes the white noise
process. 
Furthermore, we prescribe an initial condition 
$h:\Omega  \to \mathbb{R}$
and boundary conditions through a function
$F:\Omega \times \mathbb{R}^+ \to \mathbb{R}^b$.

Any numerical method used to approximate solutions for \eqref{eq:main}
involves discretizing both the spatial and time domains.
First, we describe the approximation of the space derivatives, for simplicity, we ignore the time contribution, and we will study only the following equation
\begin{equation}
\label{eq:eq9}
\nabla^2 v(\mathbf{x})=0.    
\end{equation}
Suppose that 
$\{\mathbf{x}_1,\cdots,\mathbf{x}_N\}$ be the set of all nodes in $\Omega$.
Let us define the star $\{\mathbf{x}_c,\mathbf{x}_1,\cdots,\mathbf{x}_M\}$ for every node, where $\mathbf{x}_c$
is the central node and the rest of nodes are the neighboring nodes of 
$\mathbf{x}_c$.
There are several ways to select the nodes of the star. 
In this work, we take into account the distance of the nodes from the central node and selected the ones closest to it \cite{Jensen1972}, see Figure \ref{fig:distance} for two-dimensional case.
\begin{figure}[H]
\begin{center}
{\includegraphics[width=0.45\textwidth]{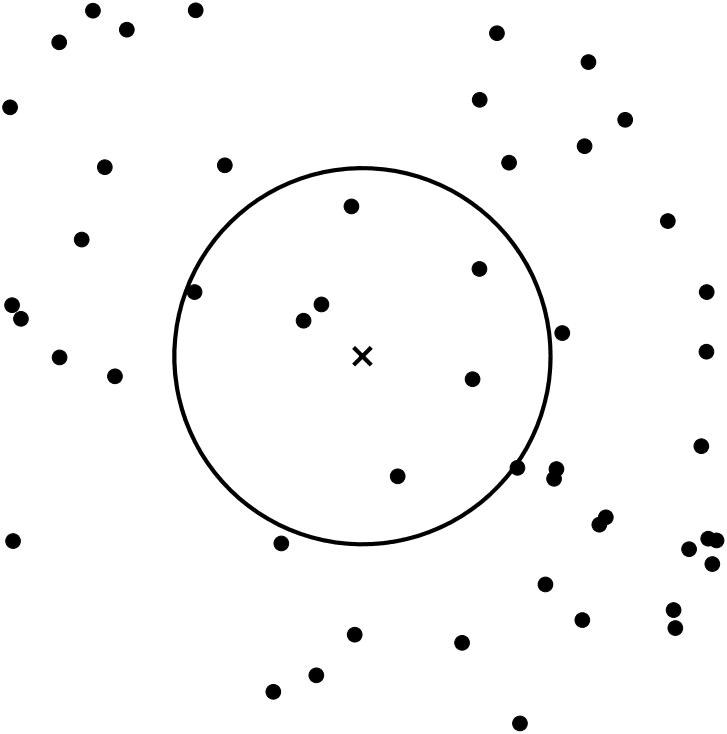}}
\end{center}
\caption{Central node with its star.}
\label{fig:distance}
\end{figure}

We can write the second-order Taylor series expansion around $\mathbf{x}_c$ 
for sufficiently differentiable function $v(\mathbf{x})$  in $\Omega$ in 
\begin{itemize}
    \item One dimension:
    \begin{subequations}
    \label{eq:eq6}
\begin{equation}
u_i=u_c
+p_i \frac{\partial u_c}{\partial x}+\frac{p_i ^2}{2}\frac{\partial^2 u_c}{\partial x^2}.
\end{equation}
 \item Two dimensions:
\begin{equation}
u_i=u_c
+p_i \frac{\partial u_c}{\partial x}
+q_i \frac{\partial u_c}{\partial y}
+\frac{p_i ^2}{2}\frac{\partial^2 u_c}{\partial x^2}
+\frac{q_i ^2}{2}\frac{\partial^2 u_c}{\partial y^2}
+p_i q_i\frac{\partial^2 u_c}{\partial x\partial y}.
\end{equation}
 \item Three dimensions:
\begin{equation}
\begin{split}
u_i&=u_c
+p_i \frac{\partial u_c}{\partial x}
+q_i \frac{\partial u_c}{\partial y}
+r_i \frac{\partial u_c}{\partial z}\\
&+\frac{p_i ^2}{2}\frac{\partial^2 u_c}{\partial x^2}
+\frac{q_i ^2}{2}\frac{\partial^2 u_c}{\partial y^2}
+\frac{r_i ^2}{2}\frac{\partial^2 u_c}{\partial z^2}
+p_i q_i\frac{\partial^2 u_c}{\partial x\partial y}
+p_i r_i\frac{\partial^2 u_c}{\partial x\partial z}
+q_i r_i\frac{\partial^2 u_c}{\partial y\partial z}.
\end{split}
\end{equation}
\end{subequations}
\end{itemize}
where $x_c$, $(x_c,y_c)$ and $(x_c,y_c,z_c)$ denote the coordinates of the central node in 1D, 2D and 3D, respectively, while $x_i$, $(x_i,y_i)$ and $(x_i,y_i,z_i)$  denote the coordinates of the $i$-th node in the star in 1D, 2D and 3D, respectively. $u_c$ and $u_i$ represent the approximation of $v_c$ and $v_i$, respectively, also
$p_i=x_i-x_c$, $q_i=y_i-y_c$ and $r_i=z_i-z_c$.

Multiplying   \eqref{eq:eq6} by 
the weight function $w_i$, which is a function of spatial increments, squaring it and then Summing them 
over all nodes in the star leads to the following formulation \cite{Benito2001} in 
\begin{itemize}
\item One dimension:
\begin{subequations}
\label{eq:eq7}
\begin{equation}
 A^{(2)}=\sum_{i=1}^M ((-u_i+u_c
+p_i \frac{\partial u_c}{\partial x}+\frac{p_i ^2}{2}\frac{\partial^2 u_c}{\partial x^2} )w_i)^2.
\end{equation}
\item Two dimensions:
\begin{equation}
 A^{(5)}=\sum_{i=1}^M ((-u_i+u_c
+p_i \frac{\partial u_c}{\partial x}
+q_i \frac{\partial u_c}{\partial y}
+\frac{p_i ^2}{2}\frac{\partial^2 u_c}{\partial x^2}
+\frac{q_i ^2}{2}\frac{\partial^2 u_c}{\partial y^2}
+p_i q_i\frac{\partial^2 u_c}{\partial x\partial y} )w_i)^2.
\end{equation}
    \item Three dimensions:
\begin{equation}
\begin{split}
 A^{(9)}&=\sum_{i=1}^M ((-u_i+u_c
+p_i \frac{\partial u_c}{\partial x}
+q_i \frac{\partial u_c}{\partial y}
+r_i \frac{\partial u_c}{\partial z}\\
&+\frac{p_i ^2}{2}\frac{\partial^2 u_c}{\partial x^2}
+\frac{q_i ^2}{2}\frac{\partial^2 u_c}{\partial y^2}
+\frac{r_i ^2}{2}\frac{\partial^2 u_c}{\partial z^2}
+p_i q_i\frac{\partial^2 u_c}{\partial x\partial y}
+p_i r_i\frac{\partial^2 u_c}{\partial x\partial z}
+q_i r_i\frac{\partial^2 u_c}{\partial y\partial z} )w_i)^2.
\end{split}
\end{equation}
\end{subequations}
\end{itemize}

Note that the weight functions are positive, symmetric, and uniformly decreasing. Specifically, they decrease as the distance from the center point to the origin increases \cite{Gavete2017}.
There are various weight functions, for example
\begin{enumerate}[label=(\alph*)]
    \item Potential weight function:
\begin{equation*}
  w_i=\frac{1}{\delta_i^n},  
\end{equation*}
\item Exponential weight function:
\begin{equation*}
  w_i=e^{-n(\delta_i)^2},  
\end{equation*}
\item Cubic spline weight function:
\begin{equation*}
    w_i=\begin{cases}
    \frac{2}{3}-4(\frac{\delta_i}{\delta_{\max}})^2+4(\frac{\delta_i}{\delta_{\max}})^3, &\text{if } \delta_i \leq \frac{1}{2}\delta_{\max},\\
    \frac{4}{3}-4(\frac{\delta_i}{\delta_{\max}})+4(\frac{\delta_i}{\delta_{\max}})^2-\frac{4}{3}(\frac{\delta_i}{\delta_{\max}})^3, &\text{if } \frac{1}{2}\delta_{max} < \delta_i \leq \delta_{\max},\\
    0, & \text{otherwise}.
    \end{cases}
\end{equation*}
\end{enumerate}
where $\delta_i$ represent the distance of the central node from the $i$-th node  in the star and  $\delta_{max}$ denote the maximum distance of the central node from all the nodes in the star.

In order to find the space derivatives, we minimize the functions defined by \eqref{eq:eq7} with respect to them. To find the optimal solution, we calculate the partial derivatives of \eqref{eq:eq7} and set them to zero. After simplification, we get the following linear equation system in 1D, 2D and 3D as follows:
\begin{itemize}
\item One dimension:
\begin{subequations}
\begin{equation}   H^{(2)}\mathbf{d}^{(2)}=\mathbf{f}^{(2)},     
\end{equation}
where $H^{(2)} \in \mathbb{R}^{2 \times 2}$, $\mathbf{d}^{(2)} \in \mathbb{R}^2$ and 
$\mathbf{f}^{(2)} \in \mathbb{R}^2$ 
are given by
\begin{align*}
H^{(2)}=&
\begin{bmatrix}
\sum\limits_{i=1}^M p_i ^2 w_i ^2 &&
\sum\limits_{i=1}^M \frac{p_i^3}{2}  w_i ^2\\ 
\sum\limits_{i=1}^M \frac{p_i^3}{2}  w_i ^2 &&
\sum\limits_{i=1}^M \frac{p_i^4}{4}  w_i ^2
\end{bmatrix},\\ 
\mathbf{d}^{(2)}=&
\begin{bmatrix}
\frac{\partial u_c}{\partial x} \\ \frac{\partial^2 u_c}{\partial x^2}
\end{bmatrix},~\mathbf{f}^{(2)}= 
\begin{bmatrix}
\sum\limits_{i=1}^M (-u_c+u_i)p_i w_i ^2 \\ 
\sum\limits_{i=1}^M (-u_c+u_i)\frac{p_i ^2}{2} w_i ^2
\end{bmatrix}.
\end{align*}
 \item Two dimensions:
    \begin{equation}
   H^{(5)}\mathbf{d}^{(5)}=\mathbf{f}^{(5)},     
    \end{equation}
where $H^{(5)} \in \mathbb{R}^{5 \times 5}$, $\mathbf{d}^{(5)} \in \mathbb{R}^5$ and 
$\mathbf{f}^{(5)} \in \mathbb{R}^5$ 
are given by
\begin{equation*}
    \begin{split}
H^{(5)}&=\begin{bmatrix}
\sum\limits_{i=1}^M p_i ^2 w_i ^2 &
\sum\limits_{i=1}^M p_i q_i w_i ^2 &
\sum\limits_{i=1}^M \frac{p_i^3}{2} w_i ^2&
\sum\limits_{i=1}^M \frac{p_i q_i^2}{2} w_i ^2&
\sum\limits_{i=1}^M p_i^2 q_i w_i ^2\\
\sum\limits_{i=1}^M p_i q_i w_i ^2 &
\sum\limits_{i=1}^M q_i^2 w_i ^2 &
\sum\limits_{i=1}^M \frac{p_i^2 q_i}{2} w_i ^2&
\sum\limits_{i=1}^M \frac{q_i^3}{2} w_i ^2&
\sum\limits_{i=1}^M p_i q_i^2 w_i ^2\\
\sum\limits_{i=1}^M \frac{p_i^3}{2} w_i ^2&
\sum\limits_{i=1}^M \frac{p_i^2 q_i}{2} w_i ^2&
\sum\limits_{i=1}^M \frac{p_i^4}{4} w_i ^2&
\sum\limits_{i=1}^M \frac{p_i^2 q_i^2}{4} w_i ^2&
\sum\limits_{i=1}^M \frac{p_i^3 q_i}{2} w_i ^2\\
\sum\limits_{i=1}^M \frac{p_i q_i^2}{2} w_i ^2&
\sum\limits_{i=1}^M \frac{q_i^3}{2} w_i ^2&
\sum\limits_{i=1}^M \frac{p_i^2 q_i^2}{4} w_i ^2&
\sum\limits_{i=1}^M \frac{q_i^4}{4} w_i ^2&
\sum\limits_{i=1}^M \frac{p_i q_i^3}{2} w_i ^2\\
\sum\limits_{i=1}^M p_i^2 q_i w_i ^2&
\sum\limits_{i=1}^M p_i q_i^2 w_i ^2&
\sum\limits_{i=1}^M \frac{p_i^3 q_i}{2} w_i ^2&
\sum\limits_{i=1}^M \frac{p_i q_i^3}{2} w_i ^2&
\sum\limits_{i=1}^M p_i^2 q_i^2 w_i ^2
\end{bmatrix},\\
\mathbf{d}^{(5)}&=
\begin{bmatrix}
\frac{\partial u_c}{\partial x} \\ 
\frac{\partial u_c}{\partial y} \\
\frac{\partial^2 u_c}{\partial x^2}\\
\frac{\partial^2 u_c}{\partial y^2}\\
\frac{\partial^2 u_c}{\partial x \partial y}
\end{bmatrix}, ~ \mathbf{f}^{(5)}= 
\begin{bmatrix}
\sum\limits_{i=1}^M (-u_c+u_i)p_i w_i ^2 \\ 
\sum\limits_{i=1}^M (-u_c+u_i)q_i w_i ^2 \\ 
\sum\limits_{i=1}^M (-u_c+u_i)\frac{p_i ^2}{2} w_i ^2\\
\sum\limits_{i=1}^M (-u_c+u_i)\frac{q_i ^2}{2} w_i ^2\\
\sum\limits_{i=1}^M (-u_c+u_i)p_i q_i  w_i ^2
\end{bmatrix}.
\end{split}
\end{equation*}
\item Three dimensions:
\begin{equation}   H^{(9)}\mathbf{d}^{(9)}=\mathbf{f}^{(9)}.     
\end{equation}
where $H^{(9)} \in \mathbb{R}^{9 \times 9}$ can be obtained with similar computations,
$\mathbf{d}^{(9)} \in \mathbb{R}^9$ and 
$\mathbf{f}^{(9)} \in \mathbb{R}^9$ 
are given by
\allowdisplaybreaks
\begin{align*}
\mathbf{d}^{(9)}=
\begin{bmatrix}
\frac{\partial u_c}{\partial x} \\ 
\frac{\partial u_c}{\partial y} \\
\frac{\partial u_c}{\partial z} \\
\frac{\partial^2 u_c}{\partial x^2}\\
\frac{\partial^2 u_c}{\partial y^2}\\
\frac{\partial^2 u_c}{\partial z^2}\\
\frac{\partial^2 u_c}{\partial x \partial y}\\
\frac{\partial^2 u_c}{\partial x \partial z}\\
\frac{\partial^2 u_c}{\partial y \partial z}
\end{bmatrix}, ~ \mathbf{f}^{(9)}= 
\begin{bmatrix}
\sum\limits_{i=1}^M (-u_c+u_i)p_i w_i ^2 \\ 
\sum\limits_{i=1}^M (-u_c+u_i)q_i w_i ^2 \\ 
\sum\limits_{i=1}^M (-u_c+u_i)r_i w_i ^2 \\ 
\sum\limits_{i=1}^M (-u_c+u_i)\frac{p_i ^2}{2} w_i ^2\\
\sum\limits_{i=1}^M (-u_c+u_i)\frac{q_i ^2}{2} w_i ^2\\
\sum\limits_{i=1}^M (-u_c+u_i)\frac{r_i ^2}{2} w_i ^2\\
\sum\limits_{i=1}^M (-u_c+u_i)p_i q_i  w_i ^2\\
\sum\limits_{i=1}^M (-u_c+u_i)p_i r_i  w_i ^2\\
\sum\limits_{i=1}^M (-u_c+u_i)q_i r_i  w_i ^2
\end{bmatrix}.
\end{align*}
\end{subequations}
\end{itemize}
Since $H^{(2)}$, $H^{(5)}$ and $H^{(9)}$ are symmetric matrices, we can write the Cholesky decomposition of it, namely $SS^T$, where $S$ is a  lower triangular matrix with positive diagonal entries. We denote the entries of the matrix $S$ by $S_{i,j}$, for $i,j=1,\cdots,C$ where $C$ is equal to 2, 5 and 9 for $H^{(2)}$, $H^{(5)}$ and $H^{(9)}$, respectively.

We can express the space derivatives through the following formulation for $l=1,\cdots,C$ \cite{Benito2001}
\begin{equation}
\label{eq:eq8}
 \mathbf{d}_l=  \frac{1}{S_{l,l}}
 (-u_c \sum_{j=1} ^{C} \beta_j R_{l,j}+
 \sum_{i=1} ^M u_i 
 (\sum_{j=1} ^{C} \alpha_{i,j} R_{l,j})
 -\sum_{n=1} ^ {C-l} S_{l+n,l}\mathbf{d}_{l+n}
 ).
\end{equation}
where the coefficients $\alpha_{i,j}$ and $\beta_j$ are given by
\begin{equation*}
\begin{split}
 \alpha_{i,1}&= p_i w_i^2,~
 \alpha_{i,2}= q_i w_i^2,~
 \alpha_{i,3}= r_i w_i^2,~
 \alpha_{i,4}= \frac{p_i^2}{2} w_i^2,~
  \alpha_{i,5}= \frac{q_i^2}{2} w_i^2,~
   \alpha_{i,6}= \frac{r_i^2}{2} w_i^2,\\
    \alpha_{i,7}&= p_i q_i w_i^2,~
     \alpha_{i,8}= p_i r_i w_i^2,~
      \alpha_{i,9}= q_i r_i w_i^2.\\
      \beta_j&=\sum_{i=1}^M \alpha_{i,j}.
   \end{split}
\end{equation*}
and the matrix $R$ belongs to $\mathbb{R}^{C \times C}$  whose $l,j$-th entry for
$l,j=1,\cdots C$
is  defined as
\allowdisplaybreaks
\begin{align*}
R_{l,j}=\begin{cases}
\frac{(-1)^{1-\delta_{l,j}}}{S_{l,l}} \sum\limits_{k=j}^{l-1}  S_{l,k}R_{k,j}, &\text{if } l>j ,\\
\frac{1}{S_{l,l}},&\text{if } l=j,\\
0, & \text{otherwise}.
\end{cases}    
\end{align*}
Here, $\delta_{l,j}$ indicates the Kronecker delta function.

Substituting the space derivatives obtained from \eqref{eq:eq8} into the \eqref{eq:eq9}
leads finally to
\begin{equation}
\label{eq:eq10}
  -\theta_c u_c+ \sum_{i=1} ^M \theta_i u_i=0.
\end{equation}
where $\theta_c=\sum\limits_{i=1} ^M \theta_i$.

In order to solve \eqref{eq:main},
the time domain $[0,T]$ is divided into time intervals $[t_k,t_{k+1}]$. It will be assumed that this partition will be regular, i.e. the  time step is $\Delta t=t_{k+1}-t_k$ for all $k$.

We then integrate \eqref{eq:main1} over time, allowing to rewrite \eqref{eq:main1} as
\begin{equation}
\label{eq:exact}
v(\mathbf{x} ,t_{k+1})-v(\mathbf{x} ,t_{k})=  \rho \int_{t_{k}} ^{t_{k+1}} \nabla^2 v(\mathbf{x} ,z)\, dz+\mu  \int_{t_{k}} ^{t_{k+1}} v(\mathbf{x},z)dW(z).
\end{equation}
Using \eqref{eq:eq10}, $u_c ^{k+1}$ is written as
\begin{equation}
\label{eq:scheme}
 u_c ^{k+1}=u_c ^k+ \rho \Delta t(-\theta_c u_c ^k+\sum_{i=1} ^M \theta_i u_i ^k ) + \mu u_c ^k(W((k+1)\Delta t)-W(k\Delta t)). 
\end{equation}

\section{Theoretical Analysis  of Stochastic Generalized Finite Difference scheme}
\label{sec:stability}

Let us start by considering the SPDE 
\begin{equation}
\label{eq:spde}
\mathcal{L}v(\mathbf{x},t)=g.
\end{equation}
where $\mathcal{L}$ denotes a  differential operator.
Equation \eqref{eq:spde} is complemented by initial
condition.
In order to approximate the solution for \eqref{eq:spde}, any numerical method involves discretizing the spatial and time domains. 
Let $\delta$  denote the maximum distance of the central node from  the nodes in the star for all the stars.
Suppose the finite difference scheme $\mathcal{L}_c ^k$ produces an approximate solution $u_c ^k$, i.e. $\mathcal{L}_c ^k u_c ^k=g_c ^k$.

We recall the important definitions, which are the key concepts of stochastic difference methods, namely consistency, stability and convergence.

\begin{definition}
A stochastic difference scheme $\mathcal{L}_c ^k u_c ^k=g_c ^k$
is called  pointwise consistent with $\mathcal{L}v=g$ at point $(\mathbf{x},t)$, in case 
\begin{equation}
\begin{split}
 \mathbb{E} \Vert (\mathcal{L}\psi-g)|_c ^k &-(\mathcal{L}_c ^k \psi(\mathbf{x}_c,t_k)-g_c ^k)\Vert^2 \to 0,\\ &\text{ for } ~ \delta \to 0,~ \Delta t \to 0,~ (\mathbf{x}_c,(k+1)\Delta t)\to (\mathbf{x},t).
 \end{split}
\end{equation}
in mean square for any continuously differentiable function $\psi=\psi(\mathbf{x},t)$.
\end{definition}

\begin{theorem}\label{theorem:consistency}
 The scheme defined in \eqref{eq:scheme} is consistent
in mean square for \eqref{eq:main}.   
\end{theorem}

\begin{proof}
\label{proof1}
Here, we consider the one-dimensional case of \eqref{eq:main} (see Appendix \ref{sec:appendix1} for proofs for the two- and three-dimensional cases). For continuously differentiable function $\psi$, we can write
\begin{equation*}
\begin{split}
    \mathcal{L}\psi|_c ^k &= \psi (x_c,(k+1)\Delta t)-\psi (x_c,k\Delta t)-
    \rho \int_{t_{k}} ^{t_{k+1}} \psi_{xx} (x_c ,l)\, dl\\
    &-\mu  \int_{t_{k}} ^{t_{k+1}} \psi(x_c,l)dW(l).
    \end{split}
\end{equation*} 
Similarly, we have
\begin{equation*}
\begin{split}
 \mathcal{L}_c ^k \psi&= \psi (x_c,(k+1)\Delta t)-\psi (x_c,k\Delta t)- \rho \Delta t(-\theta_c \psi (x_c,k\Delta t)+\sum_{i=1} ^M \theta_i \psi (x_i,k\Delta t) )\\
 &- \mu \psi (x_c,k\Delta t)(W((k+1)\Delta t)-W(k\Delta t)).  
 \end{split}
\end{equation*}
By taking the expected value of the square of  $\mathcal{L}\psi|_c ^k-   \mathcal{L}_c ^k \psi$, we get
\begin{equation*}
\begin{split}
\mathbb{E}&|\mathcal{L}\psi|_c ^k-   \mathcal{L}_c ^k \psi|^2 \\&\leq 2 \rho ^2 \mathbb{E}|
\int_{t_{k}} ^{t_{k+1}} (\psi_{xx} (x_c ,l)-(-\theta_c \psi (x_c,k\Delta t)+\sum_{i=1} ^M \theta_i \psi (x_i,k\Delta t)))\, dl|^2\\&
+2 \mu ^2 \mathbb{E}|
\int_{t_{k}} ^{t_{k+1}} (\psi(x_c,l)
-\psi (x_c,k\Delta t))
dW(l)
|^2,
\end{split}
\end{equation*}
Note that $\psi$ is deterministic. We get the following relation
\begin{equation}
\label{eq:eq2}
\begin{split}
\mathbb{E}&|\mathcal{L}\psi|_c ^k-   \mathcal{L}_c ^k \psi|^2 \\&\leq 2 \rho ^2 \mathbb{E}|
\int_{t_{k}} ^{t_{k+1}} (\psi_{xx} (x_c ,l)-(-\theta_c \psi (x_c,k\Delta t)+\sum_{i=1} ^M \theta_i \psi (x_i,k\Delta t)))\, dl|^2\\&
+2 \mu ^2 
\int_{t_{k}} ^{t_{k+1}} |\psi(x_c,l)
-\psi (x_c,k\Delta t)|^2
dl.
\end{split}
\end{equation}
when $\Delta t \to 0$, the last term tends to zero.
Next,  we prove that the first term tends to zero as $\delta \to 0$ and $\Delta t\to 0$.  The main idea is  to use the Taylor series expansion. The approximation for $\psi_{xx} (x_c ,k\Delta t)$ is obtained in a similar
manner that is introduced in Section \ref{sec:methodology}, see \cite{Benito2007} for details. 
\begin{equation}
\label{eq:eq1}
 \psi_{xx} (x_c ,k\Delta t)-(-\theta_c \psi (x_c,k\Delta t)+\sum_{i=1} ^M \theta_i \psi (x_i,k\Delta t)) = 
 \begin{bmatrix}
0 && 1
\end{bmatrix}
(\mathbf{d}^{(2)}-(H^{(2)})^{-1}\mathbf{f}^{(2)}).
\end{equation}
Suppose  that we use higher order derivatives in the Taylor series expansion. By  minimizing the norm $A^{(2)}(\psi)$ with respect to  $\frac{\partial \psi_c}{\partial x}$ and $\frac{\partial^2 \psi_c}{\partial x^2}$, we  get  
\begin{equation*}
 H^{(2)}  \mathbf{d}^{(2)}=  \tilde{\mathbf{f}}^{(2)}.
\end{equation*}
We substitute $\mathbf{d}^{(2)}$ in
\eqref{eq:eq1}, we have
\begin{equation*}
\begin{split}
 \psi_{xx} (x_c ,k\Delta t)&-(-\theta_c \psi (x_c,k\Delta t)+\sum_{i=1} ^M \theta_i \psi (x_i,k\Delta t)) \\
 &= 
 \begin{bmatrix}
0 && 1
\end{bmatrix}
(H^{(2)})^{-1}(\tilde{\mathbf{f}}^{(2)}-\mathbf{f}^{(2)})\\
&= 
 \begin{bmatrix}
0 && 1
\end{bmatrix}
(H^{(2)})^{-1}
\begin{bmatrix}
- \sum\limits_{i=1}^M (\frac{1}{3!}(p_i \frac{\partial \psi_c}{\partial x})^3
+\frac{1}{4!}(p_i \frac{\partial \psi_c}{\partial x})^4+\cdots)p_i w_i^2\\
- \sum\limits_{i=1}^M (\frac{1}{3!}(p_i \frac{\partial \psi_c}{\partial x})^3
+\frac{1}{4!}(p_i \frac{\partial \psi_c}{\partial x})^4+\cdots)
\frac{p_i^2}{2} 
w_i^2
\end{bmatrix}\\
&= -
\sum\limits\limits_{i=1}^M  \frac{1}{3!} (p_i\frac{\partial \psi_c}{\partial x})^3 \theta_i
-
\sum\limits_{i=1}^M  \frac{1}{4!} (p_i \frac{\partial \psi_c}{\partial x})^4 \theta_i -\cdots\\
&= -
\sum\limits_{i=1}^M  \frac{p_i^3}{3!} \frac{\partial^3 \psi_c}{\partial x^3} \theta_i
-
\sum\limits_{i=1}^M  \frac{p_i^4}{4!} \frac{\partial^4 \psi_c}{\partial x^4} \theta_i -\cdots.
\end{split}
\end{equation*}
Therefore, as $\delta\to 0$ and $\Delta t\to 0$, the
first term in \eqref{eq:eq2} tends to zero. Finally, this proves that \eqref{eq:scheme} is consistent
in mean square with \eqref{eq:main}.
\end{proof}

\begin{definition}
Given a stochastic difference scheme, it is stable with respect to a norm in mean square supposing that there exist non-negative constants $C$ and $\gamma$ and positive constants $\delta'$ and $\Delta t'$ verifying the relation
\begin{equation}
\begin{split}
\mathbb{E} \Vert u^{k+1}\Vert ^2&\leq 
Ce^{\gamma t}\mathbb{E} \Vert u^{0}\Vert ^2,\\ &\forall t=(k+1)\Delta t,~0\leq \delta\leq \delta',~ 0\leq \Delta t \leq \Delta t'.
\end{split}
\end{equation}
\end{definition}

\begin{theorem}\label{theorem:stable}
 The scheme defined in \eqref{eq:scheme} is stable with respect to sup-norm in mean square
 for $0\leq \rho \Delta t \max\limits_{c} \theta_c \leq 1$
 for \eqref{eq:main}.
\end{theorem}
\begin{proof}
We take the expected value of the square of  \eqref{eq:scheme}.  Knowing that the Wiener increments are independent, we obtain
\allowdisplaybreaks
\begin{align*}
\mathbb{E} |u_c ^{k+1}|^2&= \mathbb{E} |u_c ^k+ \rho \Delta t(-\theta_c u_c ^k+\sum_{i=1} ^M \theta_i u_i ^k )|^2 + \mu^2 \Delta t \mathbb{E} |u_c ^k|^2 \\
&= \mathbb{E} |(1-\rho \Delta t\theta_c)u_c ^k+\rho \Delta t\sum_{i=1} ^M \theta_i u_i ^k )|^2 + \mu^2 \Delta t \mathbb{E} |u_c ^k|^2 \\
&= (1-\rho \Delta t\theta_c)^2 \mathbb{E}|u_c ^k|^2+(\rho \Delta t)^2\sum_{i=1} ^M \theta_i^2  \mathbb{E}|u_i ^k|^2\\
&+2|1-\rho \Delta t\theta_c||\rho \Delta t|
\sum_{i=1} ^M |\theta_i| \mathbb{E} |u_i ^ku_c ^k|
+2 (\rho \Delta t)^2\sum_{i=1} ^M \sum_{i'=i+1} ^M |\theta_i \theta_{i'}| \mathbb{E} |u_i ^k u_{i'} ^k|
\\
&+ \mu^2 \Delta t \mathbb{E} |u_c ^k|^2\\
&\leq ((|1-\rho \Delta t\theta_c|+|\rho \Delta t\theta_c|)^2+ \mu^2 \Delta t)\sup\limits_{c} \mathbb{E} |u_c ^k|^2.
\end{align*}
under the condition 
$0\leq \rho \Delta t \max\limits_{c} \theta_c \leq 1$
, there holds the following relation
\begin{equation*}
\mathbb{E} |u_c ^{k+1}|^2\leq 
(1+ \mu^2 \Delta t)\sup\limits_{c} \mathbb{E} |u_c ^k|^2,
\end{equation*}
Hence, we have
\begin{equation}
\label{eq:eq3}
\sup\limits_{c} \mathbb{E} |u_c ^{k+1}|^2\leq 
(1+ \mu^2 \Delta t)\sup\limits_{c} \mathbb{E} |u_c ^k|^2,
\end{equation}
Substituting $(k+1)\Delta t=t$ into \eqref{eq:eq3}, we obtain
\begin{equation*}
\begin{split}
 \mathbb{E} \Vert u^{k+1}\Vert _{\infty}^2&\leq 
(1+ \frac{\mu^2 t}{k+1})\mathbb{E}\Vert u^{k}\Vert _{\infty}^2\\
&\leq 
(1+ \frac{\mu^2 t}{k+1})^{k+1}\mathbb{E}\Vert u^{0}\Vert _{\infty}^2,
\end{split}
\end{equation*}
One can write:
\begin{equation*}
 \mathbb{E} \Vert u^{k+1}\Vert _{\infty}^2\leq 
e^{\mu^2 t}\mathbb{E} \Vert u^{0}\Vert _{\infty}^2.
 \end{equation*}
Therefore,  the proposed scheme is conditionally stable.
\end{proof}

\begin{definition}
A stochastic difference scheme $\mathcal{L}_c ^k u_c ^k=g_c ^k$ is called convergent in mean square at time
$t$ for solving  $\mathcal{L}v=g$, in case
\begin{equation}
\mathbb{E} \Vert u^{k+1}- v^{k+1} \Vert^2 \to 0, ~\text{as } (k+1)\Delta t \to t.
\end{equation}
and $(k+1)\Delta t= t$, $\delta \to 0$ and $\Delta t\to 0$.
\end{definition}

\begin{theorem}
 The scheme defined in \eqref{eq:scheme} is  convergent for the   sup-norm for 
 $0\leq \rho \Delta t \max\limits_{c} \theta_c \leq 1$
 for solving \eqref{eq:main}.
\end{theorem}

\begin{proof}
Using Theorems \ref{theorem:consistency} and \ref{theorem:stable}, and due to the stochastic Lax-Richtmyer theorem \cite{Roth1989}, the scheme defined in \eqref{eq:scheme} is  convergent  under the assumption that 
 $0\leq \rho \Delta t \max\limits_{c} \theta_c \leq 1$.
\end{proof}

\section{Numerics}
\label{sec:numerics}
In this section, we present numerical results in order to validate the theoretical results.
We perform the convergence analysis for the one- , two- and three-dimensional test cases.

For all test problems presented in this article, we perform 1000 realizations, and we use
the following potential weight function
\begin{equation*}
  w_i=\frac{1}{\delta_i^3}.  
\end{equation*}
The $L ^2$-error and $L ^{\infty}$-error are evaluated as:
\begin{align*}
L ^2\text{-error}&= \sqrt{ \frac{1}{N_t N} \sum\limits_{k=1} ^{N_t}
\sum\limits_{c=1} ^N 
|v(\mathbf{x}_c,t_k)-\mathbb{E}(u_{c} ^k)|^2},\\
L ^{\infty}\text{-error}&=\max\limits_c |v(\mathbf{x}_c,t_k)-\mathbb{E}(u_{c} ^k)|.
\end{align*}
where  $N_t$ represents the total number of time steps. Moreover, $\mathbb{E}(u_{c} ^k)$ and $v(\mathbf{x}_c,t_k)$ denote the expectation of numerical solution and analytical solution at  node $\mathbf{x}_c$ and time $t_k$, respectively.

\subsection{1D Numerical Tests}
Let us consider  the one-dimensional stochastic diffusion equation
\begin{align}
\label{eq:1d}
\begin{cases}
v_t(x ,t)&=  \rho  v_{xx}(x ,t)+\mu v(x,t)dW(t),~ x\in [0,1],~ t \in [0,1],\\
v(x,0)&=\sin(\pi x),~ x\in [0,1],\\
v(0,t)&=v(1,t)=0,~ t \in [0,1].
\end{cases}
\end{align}
It is easy to see that \eqref{eq:1d} has the following exact
expected solution
\begin{equation*}
 v(x,t)=e^{-\rho \pi^2 t}  \sin(\pi x). 
\end{equation*}
The space domain is discretized using random points given in Figure \ref{fig:1d}. The refined mesh is generated by adding midpoints. We set $\rho=0.005$ and $\mu=0.1$. The $L ^2$-error and $L ^{\infty}$-error for \eqref{eq:1d} can be found in Table \ref{tab:error_1d}. As expected, the errors decrease.
Figure \ref{fig:sol_37} shows the mean solution and the analytical solution for $N=37$ at final time $T=1$.
Further, in Figure \ref{fig:unstable_37} we consider the particular case where the stability condition is not satisfied when $N=37$ and confirm the stability conditions give in Theorem \ref{theorem:stable}.

\begin{figure}[H]
\begin{center}
{\includegraphics[width=0.45\textwidth]{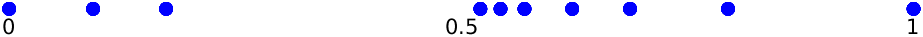}}
\end{center}
\caption{One dimension grid of nodes.}
\label{fig:1d}
\end{figure}

\begin{table}[H]
\centering
  \caption{Convergence study for the stochastic diffusion equation in one dimension.}\label{tab:error_1d}
\begin{tabular}{c   c c} 
\hline
$N$  & $L ^2$-error & $L ^{\infty}$-error\\
\hline
10  &3.4524e-03 &1.4585e-02 \\
19  &4.2493e-04 & 2.2295e-03\\
37  &4.0774e-04 & 2.0976e-03\\
\hline
\end{tabular}
\end{table}

\begin{figure}[H]
\begin{center}
{\includegraphics[width=0.45\textwidth]{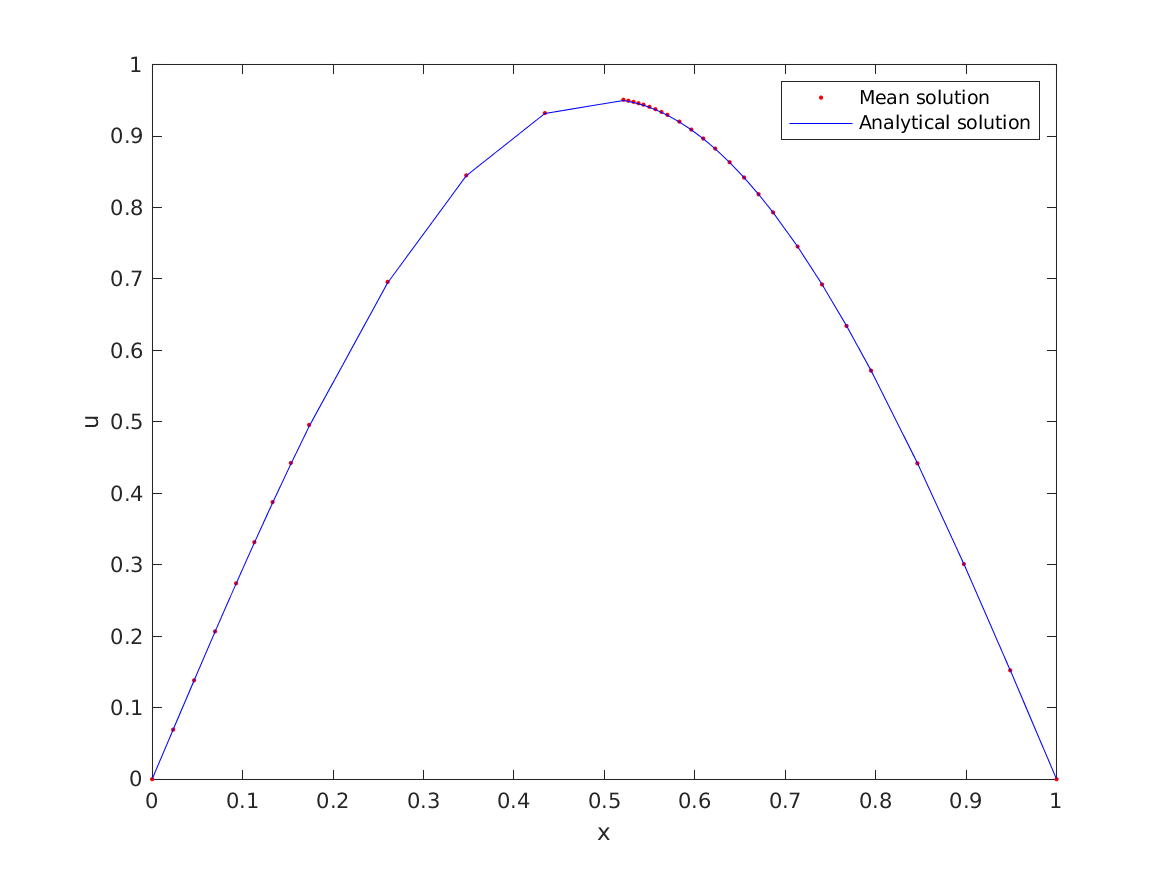}}
\end{center}
\caption{Plot of the mean solution and the analytical solution for the stochastic diffusion equation in one dimension.}
\label{fig:sol_37}
\end{figure}

\begin{figure}[H]
\begin{center}
{\includegraphics[width=0.45\textwidth]{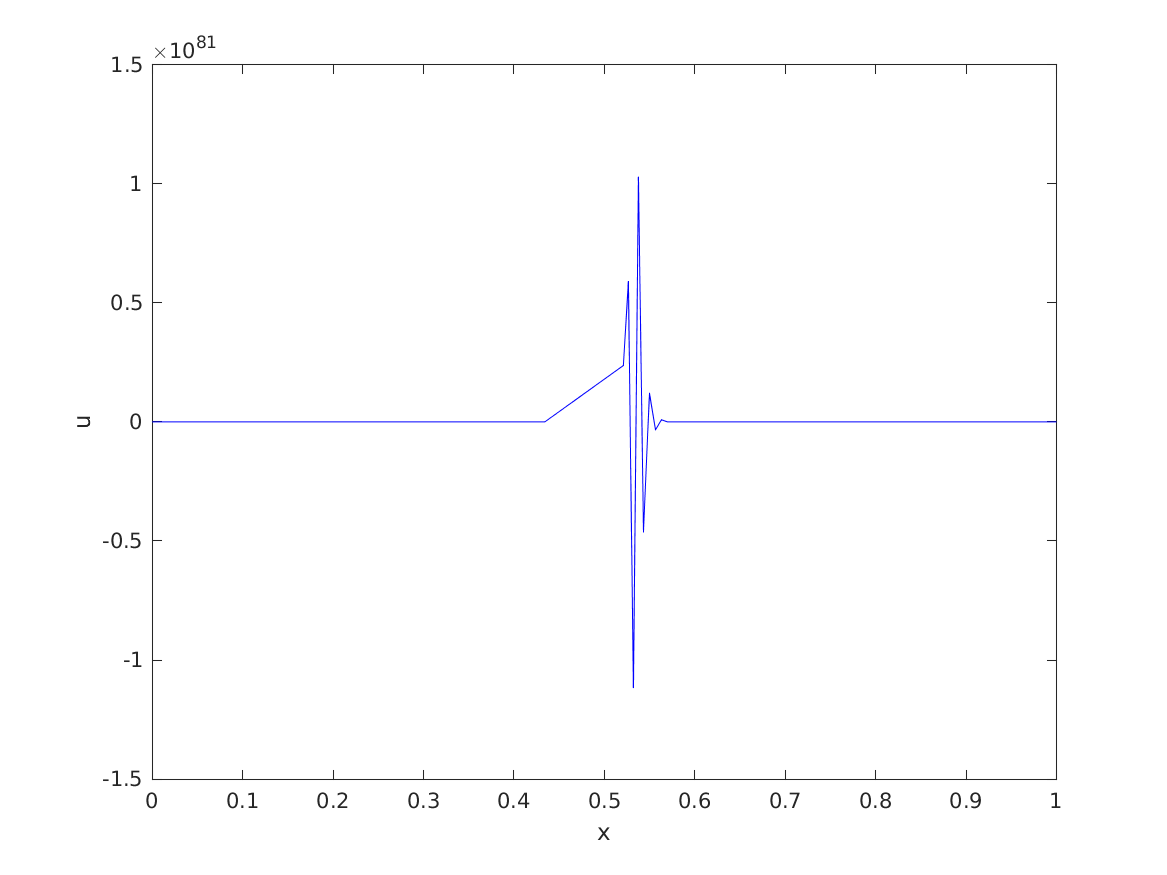}}
\end{center}
\caption{Plot of the mean solution for the stochastic diffusion equation for one dimension with $\rho=0.01$ and $\mu=0.1$.}
\label{fig:unstable_37}
\end{figure}

\subsection{2D Numerical Tests}
We have tested our strategy on \eqref{eq:main1} in two dimensions with two different clouds of points, see Figures \ref{fig:cloud1-2-3} and \ref{fig:cloud4-5-6}. 
The computational domain is a square $\Omega=[0,1] \times [0,1]$.
The initial condition is
\begin{equation*}
v(\mathbf{x},0)=\sin(\frac{x+y}{2} ),~\mathbf{x}\in \Omega, ~t>0,  
\end{equation*}
and the boundary condition is
\begin{align*}
v(\mathbf{x},t)&=e^{-\frac{\rho t}{2}}\sin(\frac{x+y}{2} ),~\mathbf{x}  \in \partial \Omega, ~t>0.    
\end{align*}
We use $\rho=0.01$. We can see in Tables \ref{tab:error_123} and \ref{tab:error_456} the errors obtained for \ref{fig:cloud1-2-3} and \ref{fig:cloud4-5-6}, respectively.
We can see that the error decay fulfils the expected behavior.
We also compare the mean and analytical solutions for cloud 3 in Figure \ref{fig:sol_289}. 
Figure \ref{fig:unstable_289} shows that the scheme is not stable under the assumption that $\rho=0.1$ for cloud 3.

\begin{figure}[H]
\begin{center}
\subfigure[Cloud 1]{\includegraphics[width=0.45\textwidth]{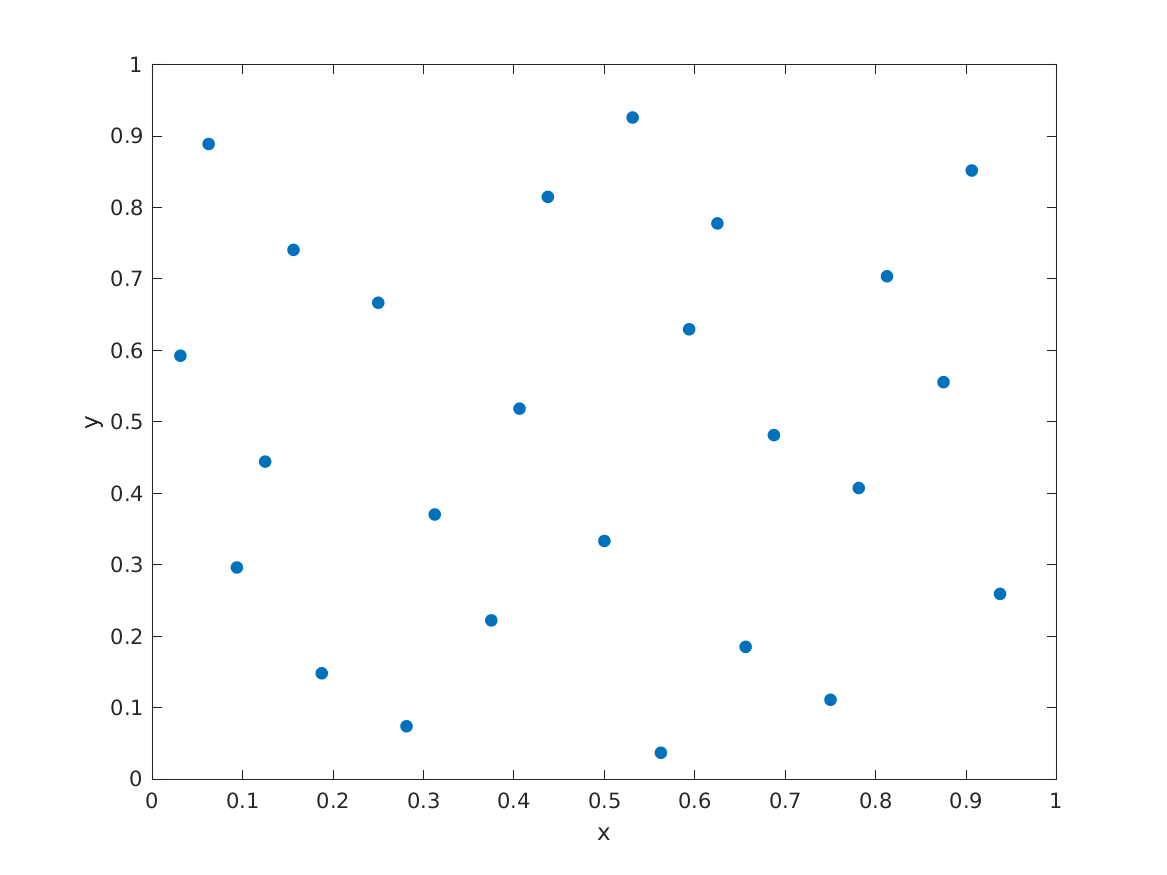}}
\subfigure[Cloud 2]{\includegraphics[width=0.45\textwidth]{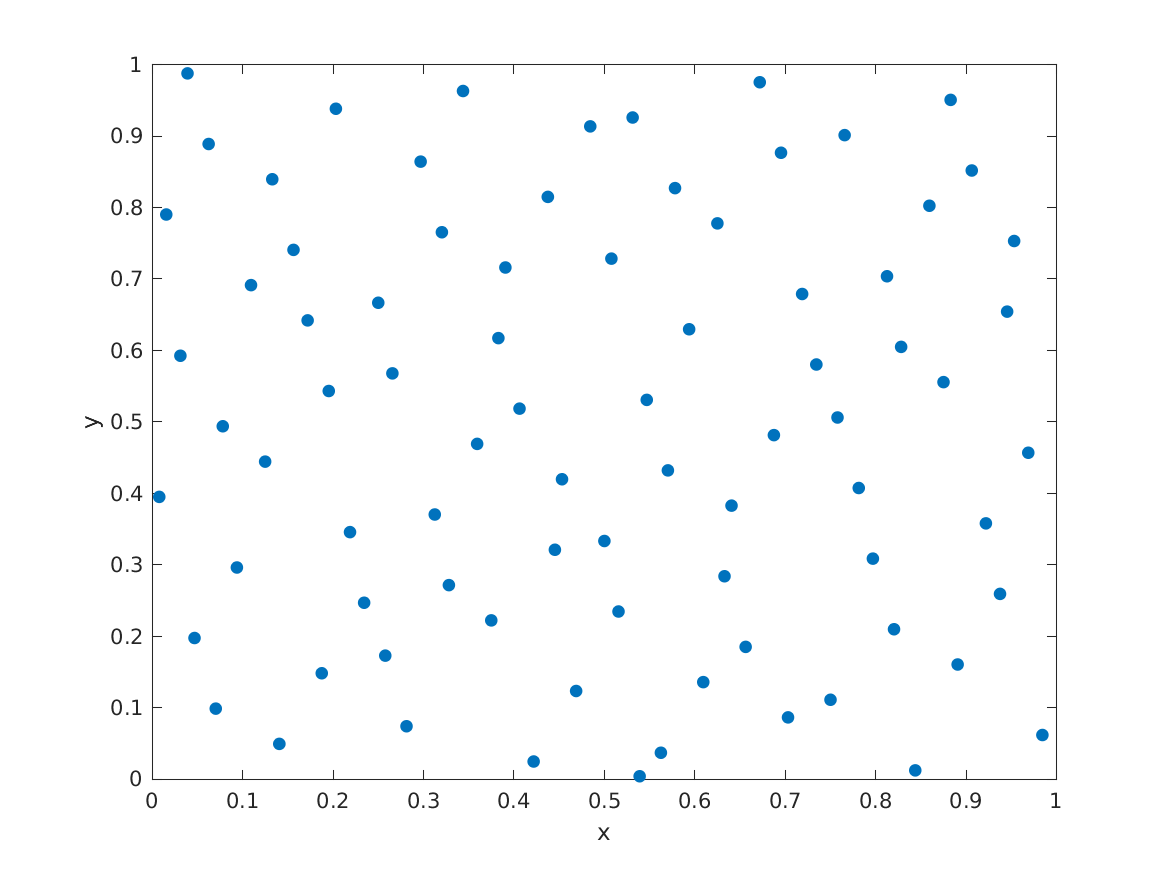}}
\subfigure[Cloud 3]{\includegraphics[width=0.45\textwidth]{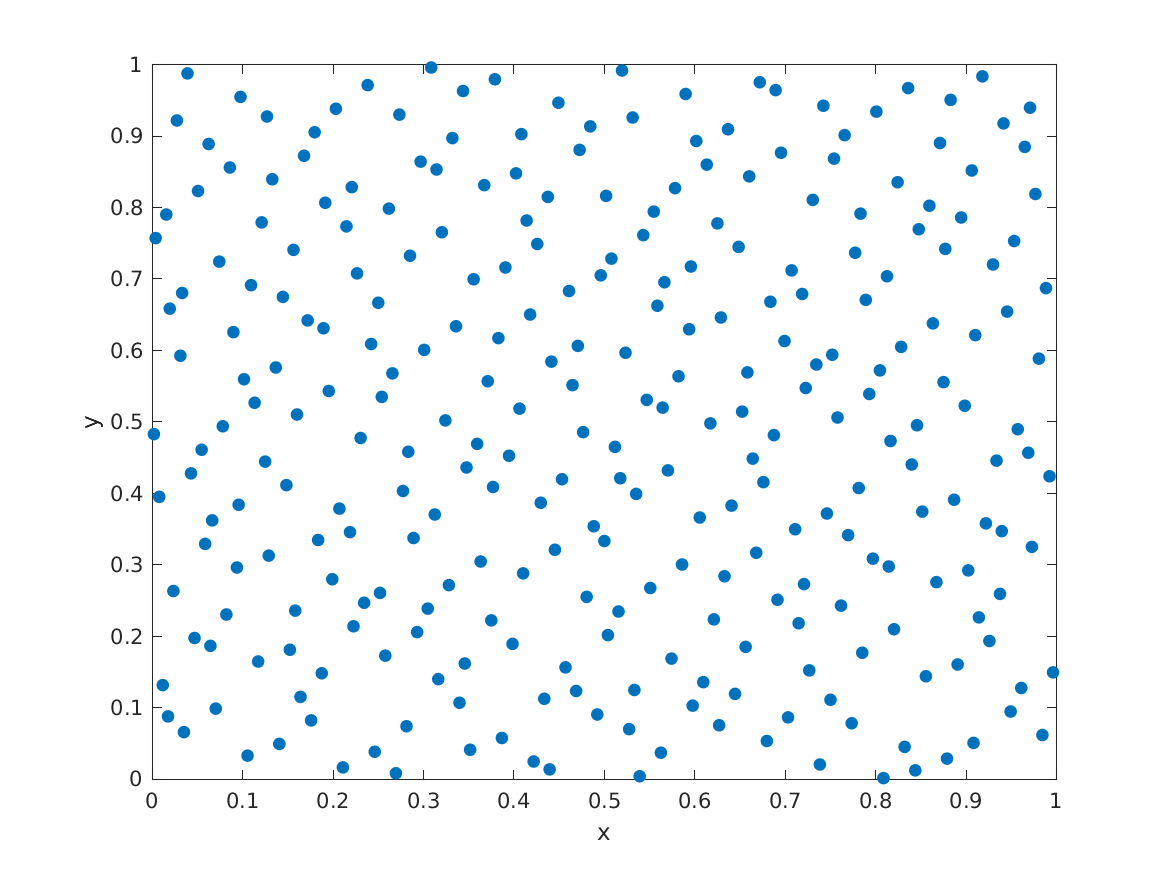}}
\end{center}
\caption{Irregular clouds of points.}
\label{fig:cloud1-2-3}
\end{figure}

\begin{figure}[H]
\begin{center}
\subfigure[Cloud 4]{\includegraphics[width=0.45\textwidth]{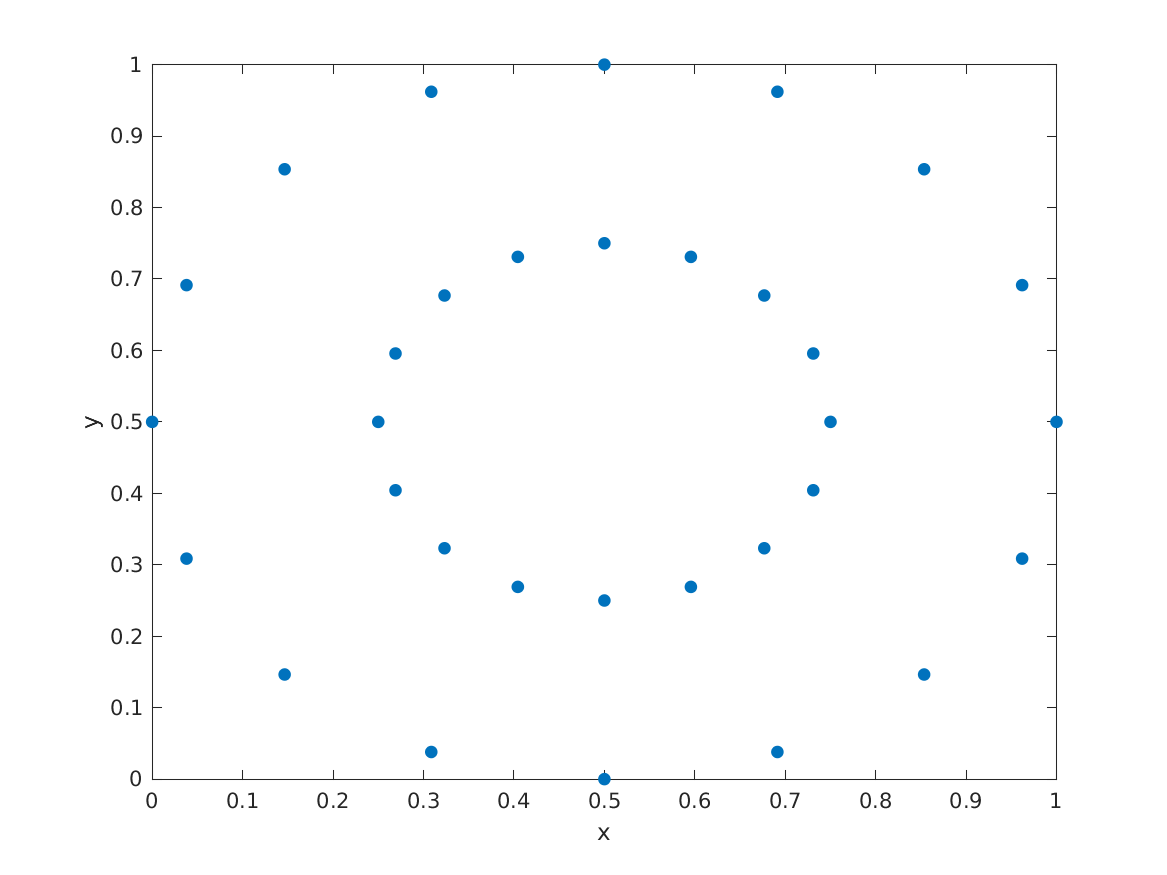}}
\subfigure[Cloud 5]{\includegraphics[width=0.45\textwidth]{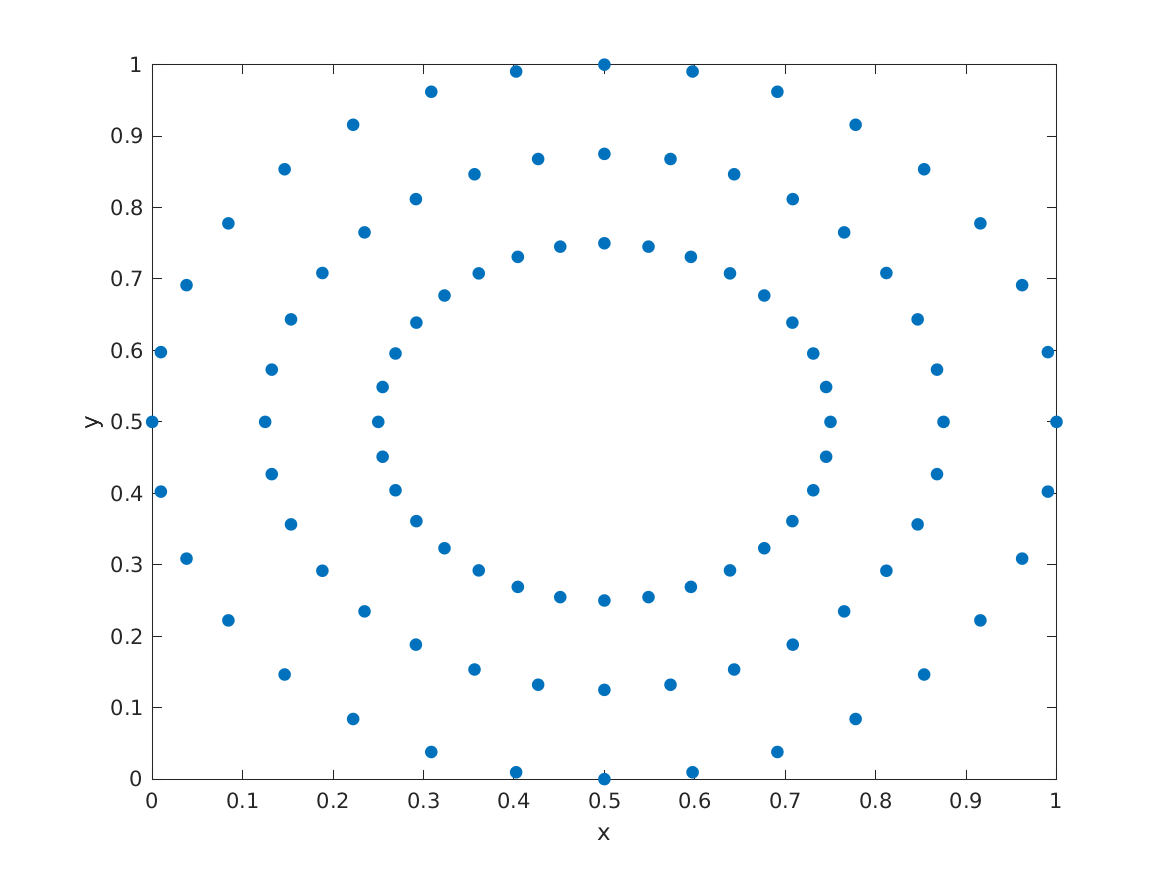}}
\subfigure[Cloud 6]{\includegraphics[width=0.45\textwidth]{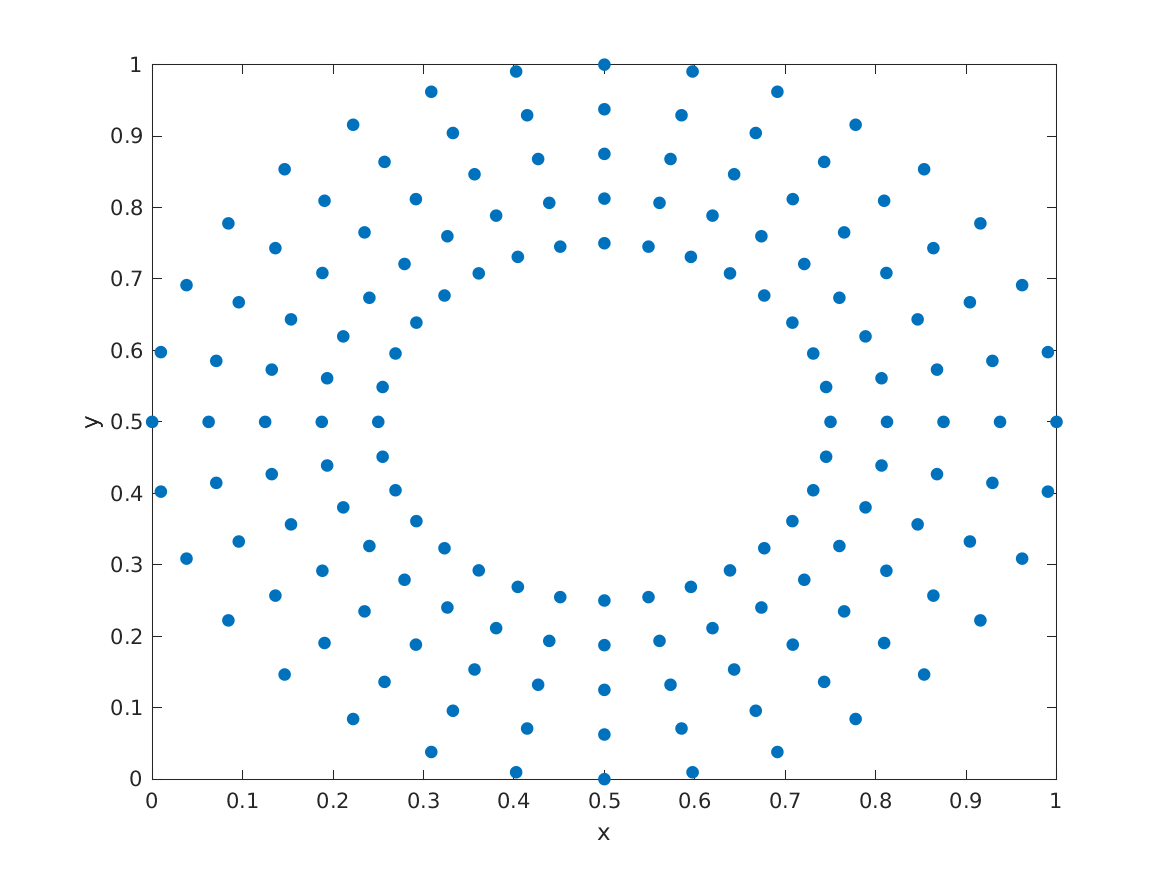}}
\end{center}
\caption{Irregular clouds of points.}
\label{fig:cloud4-5-6}
\end{figure}

\begin{table}[H]
\centering
  \caption{Convergence study for the stochastic diffusion equation for Figure \ref{fig:cloud1-2-3}.}\label{tab:error_123}
\begin{tabular}{c c  c} 
\hline
\# cloud of points  & $L ^2$-error & $L ^{\infty}$-error\\
\hline
Cloud 1  & 5.9423e-04&
1.5984e-03\\
Cloud 2   & 4.0679e-04&
9.7587e-04\\
Cloud 3  & 2.3432e-04&
7.7837e-04\\
\hline
\end{tabular}
\end{table}

\begin{table}[H]
\centering
  \caption{Convergence study for the stochastic diffusion equation for Figure \ref{fig:cloud4-5-6}.}\label{tab:error_456}
\begin{tabular}{c   c c} 
\hline
\# cloud of points   & $L ^2$-error & $L ^{\infty}$-error\\
\hline
Cloud 4   
&6.1363e-04  & 1.7569e-03\\
Cloud 5 
&   3.9460e-04& 1.1066e-03\\
Cloud 6 
& 3.3587e-04  &  8.7329e-04\\  
\hline
\end{tabular}
\end{table}

\begin{figure}[H]
\begin{center}
\subfigure[Mean solution]
{\includegraphics[width=0.45\textwidth]{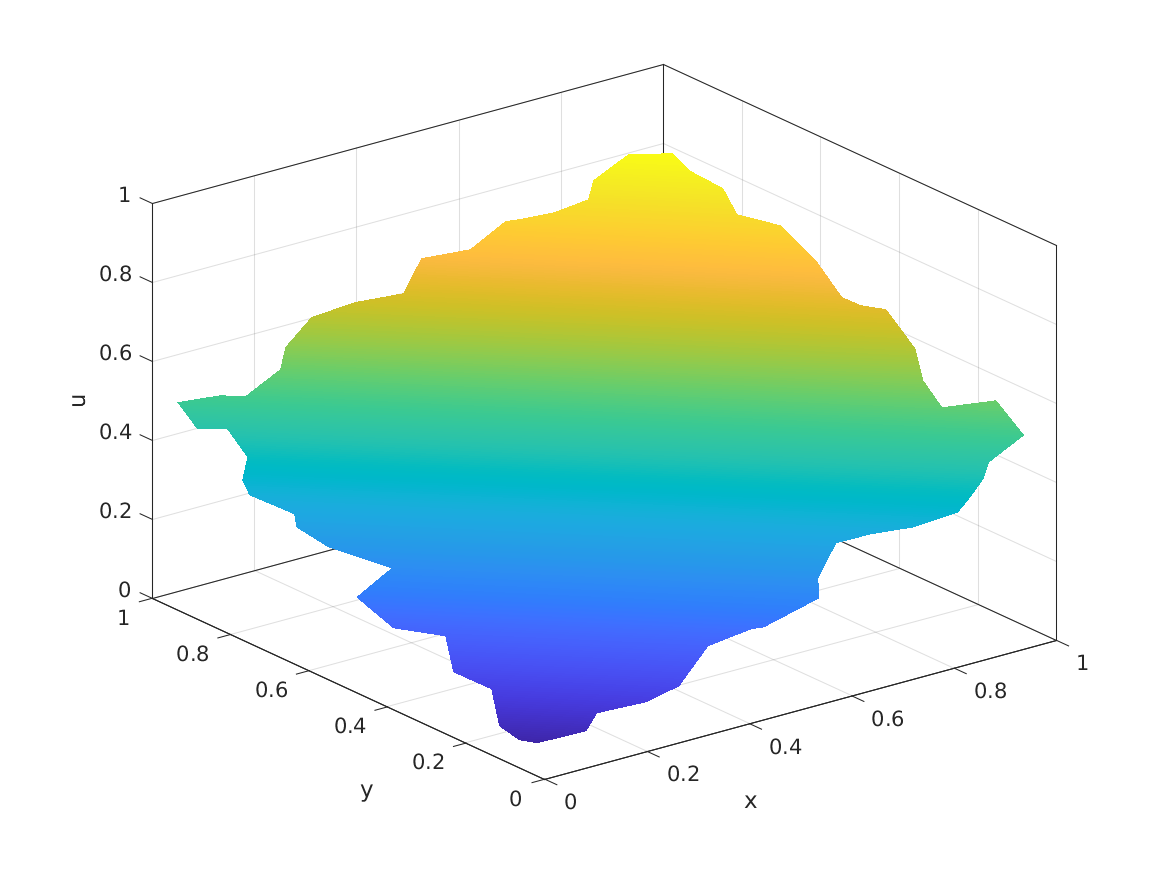}}
\subfigure[Analytical solution]
{\includegraphics[width=0.45\textwidth]{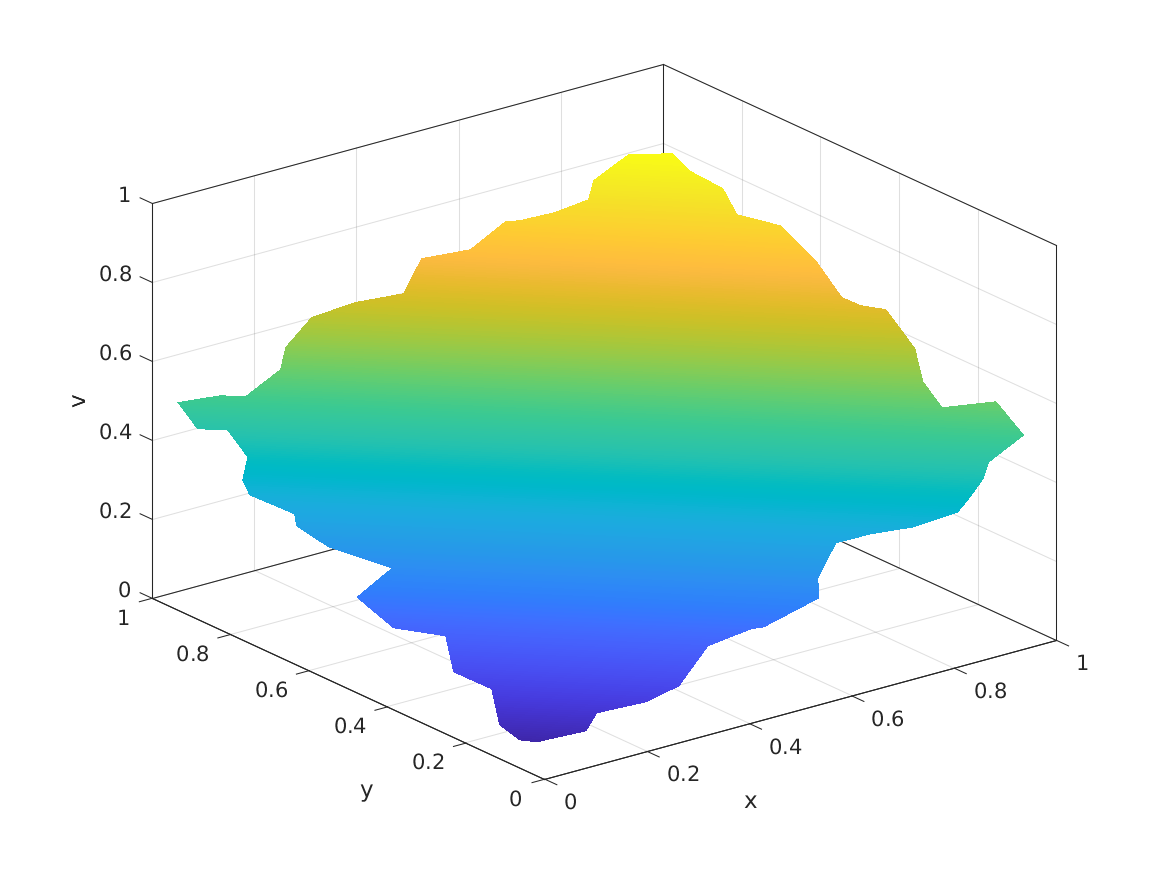}}
\end{center}
\caption{Plot of the mean solution and the analytical solution for the stochastic diffusion equation for cloud 3.}
\label{fig:sol_289}
\end{figure}

\begin{figure}[H]
\begin{center}
{\includegraphics[width=0.45\textwidth]{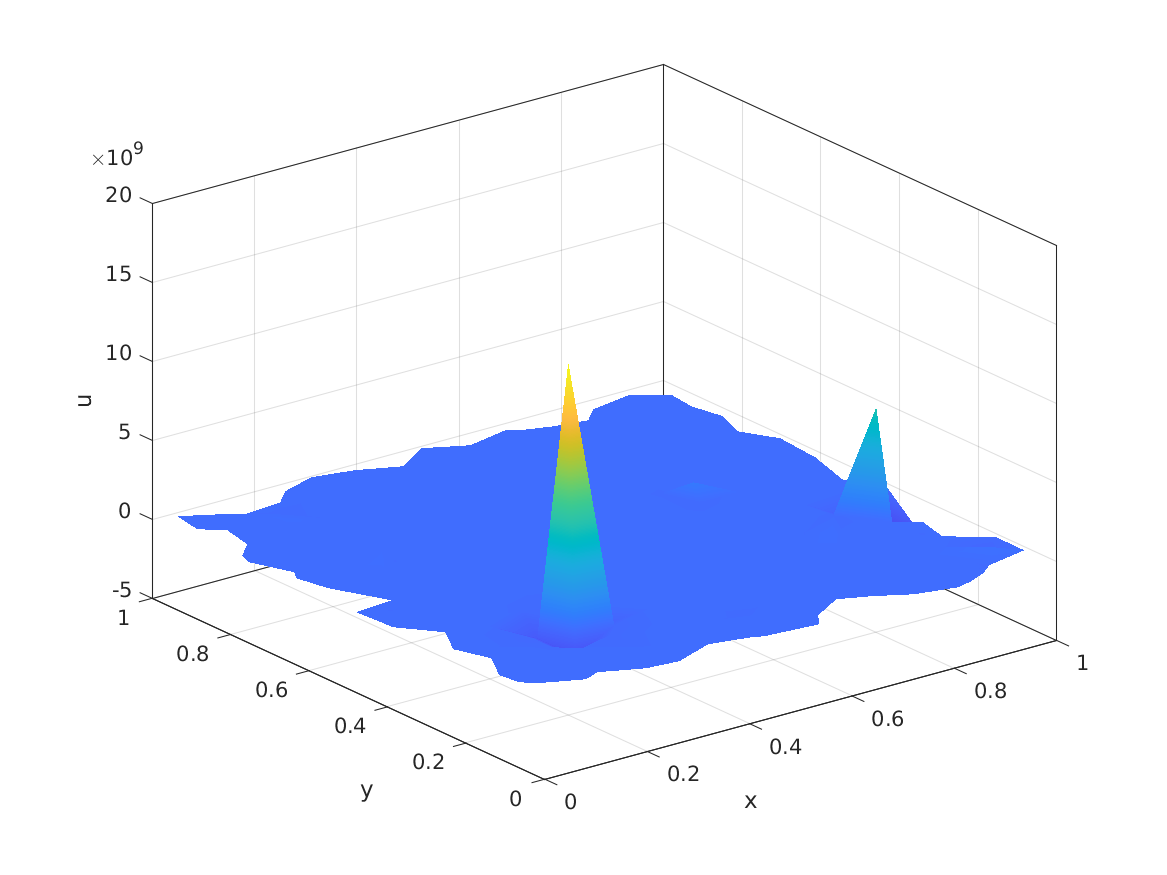}}
\end{center}
\caption{Plot of the mean solution for the stochastic diffusion equation for cloud 3 with $\rho=0.1$.}
\label{fig:unstable_289}
\end{figure}

\subsection{3D Numerical Tests}
Here, we consider the three-dimensional case. The problem is defined on $[0,1]\times [0,1] \times[0,1]$ for $T=1$. Let us consider \eqref{eq:main1} with an initial condition 
\begin{equation*}
v(\mathbf{x},0)=\sin(\pi x)
\sin(\pi y) \sin(\pi z),~\mathbf{x}\in \Omega, ~
\end{equation*}
and boundary condition
\begin{align*}
v(\mathbf{x},t)&=e^{-3\rho\pi ^2 t}\sin(\pi x)
\sin(\pi y) \sin(\pi z),~\mathbf{x}  \in \partial \Omega, ~t>0.    
\end{align*}
We set $\rho=1$ and $\mu=0.5$. In this example, we use  regular grids. The errors are shown in Table \ref{tab:error_3d}. We observe that the error decreases consistently under mesh refinement.

\begin{table}[H]
\centering
  \caption{Convergence study for the stochastic diffusion equation in three dimensions.}\label{tab:error_3d}
\begin{tabular}{c   c c} 
\hline
$N$  & $L ^2$-error & $L ^{\infty}$-error\\
\hline
64   
&1.5295e-03 & 1.7515e-02\\
125 
&9.1235e-04  &1.4346e-02 \\
216 
&  6.1913e-04 & 1.0207e-02\\  
\hline
\end{tabular}
\end{table}

\section{Conclusion}
\label{sec:conclusion}
In this paper, we have developed the generalized finite difference method  for approximating the solution of stochastic diffusion equations. Additionally, we were studied consistency, stability, and convergence analyses. The proposed approach is tested on one, two and three space dimensions and show promising results.

In future work, we will extend our investigation.
First, we will consider the general form of the SPDEs. One of key parameters of the method is the weight function. Furthermore, one can find the weights using similar strategy to the one proposed in 
in \cite{nn2024}.

\begin{appendices}
\section{}\label{sec:appendix1}
\subsection{Proof for Two-dimensional Case}
\label{sec:appendix2}
We provide a similar proof to that of Theorem \ref{theorem:consistency} for one-dimensional case.
For continuously differentiable function $\psi$, we can write
\begin{equation*}
\begin{split}
    \mathcal{L}\psi|_c ^k &= \psi (\mathbf{x}_c,(k+1)\Delta t)-\psi (\mathbf{x}_c,k\Delta t)-
    \rho \int_{t_{k}} ^{t_{k+1}} \nabla^2 \psi (\mathbf{x}_c ,l)\, dl\\
    &-\mu  \int_{t_{k}} ^{t_{k+1}} \psi(\mathbf{x}_c,l)dW(l).
    \end{split}
\end{equation*} 
We get the following relation for $\mathcal{L}_c ^k \psi$
\begin{equation*}
\begin{split}
 \mathcal{L}_c ^k \psi&= \psi (\mathbf{x}_c,(k+1)\Delta t)-\psi (\mathbf{x}_c,k\Delta t)- \rho \Delta t(-\theta_c \psi (\mathbf{x}_c,k\Delta t)+\sum_{i=1} ^M \theta_i \psi (\mathbf{x}_i,k\Delta t) )\\
 &- \mu \psi (\mathbf{x}_c,k\Delta t)(W((k+1)\Delta t)-W(k\Delta t)).  
 \end{split}
\end{equation*}
By taking the expected value of the square of  $\mathcal{L}\psi|_c ^k-   \mathcal{L}_c ^k \psi$, we get
\begin{equation*}
\begin{split}
\mathbb{E}&|\mathcal{L}\psi|_c ^k-   \mathcal{L}_c ^k \psi|^2 \\&\leq 2 \rho ^2 \mathbb{E}|
\int_{t_{k}} ^{t_{k+1}} (\nabla^2 \psi (\mathbf{x}_c ,l)-(-\theta_c \psi (\mathbf{x}_c,k\Delta t)+\sum_{i=1} ^M \theta_i \psi (\mathbf{x}_i,k\Delta t)))\, dl|^2\\&
+2 \mu ^2 \mathbb{E}|
\int_{t_{k}} ^{t_{k+1}} (\psi(\mathbf{x}_c,l)
-\psi (\mathbf{x}_c,k\Delta t))
dW(l)
|^2,
\end{split}
\end{equation*}
Reminding that $\psi$ is deterministic, we have
\begin{equation}
\label{eq:eq12}
\begin{split}
\mathbb{E}&|\mathcal{L}\psi|_c ^k-   \mathcal{L}_c ^k \psi|^2 \\&\leq 2 \rho ^2 \mathbb{E}|
\int_{t_{k}} ^{t_{k+1}} (\nabla^2 \psi (\mathbf{x}_c ,l)-(-\theta_c \psi (\mathbf{x}_c,k\Delta t)+\sum_{i=1} ^M \theta_i \psi (\mathbf{x}_i,k\Delta t)))\, dl|^2\\&
+2 \mu ^2 
\int_{t_{k}} ^{t_{k+1}} |\psi(\mathbf{x}_c,l)
-\psi (\mathbf{x}_c,k\Delta t)|^2
dl.
\end{split}
\end{equation}
As before, when $\Delta t \to 0$, the last term tends to zero.
Next, we prove that the first term tends to zero as $\delta \to 0$ and $\Delta t\to 0$. We can write 
\begin{equation}
\label{eq:eq11}
\begin{split}
 \nabla^2 \psi (\mathbf{x}_c ,k\Delta t)&-(-\theta_c \psi (\mathbf{x}_c,k\Delta t)+\sum_{i=1} ^M \theta_i \psi (\mathbf{x}_i,k\Delta t)) \\
 &= 
 \begin{bmatrix}
0 && 0 && 1 && 1 && 0
\end{bmatrix}
(\mathbf{d}^{(5)}-(H^{(5)})^{-1}\mathbf{f}^{(5)}).
\end{split}
\end{equation}
Suppose  that we use higher order derivatives in the Taylor series expansion. By  minimizing the norm $A^{(5)}(\psi)$ with respect to  the space variables, we  get  
\begin{equation*}
 H^{(5)}  \mathbf{d}^{(5)}=  \tilde{\mathbf{f}}^{(5)}.
\end{equation*}
By substituting $\mathbf{d}^{(5)}$ in
\eqref{eq:eq11}, we have
\begin{equation*}
\begin{split}
 \nabla^2 & \psi (\mathbf{x}_c ,k\Delta t)-(-\theta_c \psi (\mathbf{x}_c,k\Delta t)+\sum_{i=1} ^M \theta_i \psi (\mathbf{x}_i,k\Delta t)) \\
 &= 
 \begin{bmatrix}
0 && 0 && 1 && 1 && 0
\end{bmatrix}
(H^{(5)})^{-1}(\tilde{\mathbf{f}}^{(5)}-\mathbf{f}^{(5)})\\
&= 
 \begin{bmatrix}
0 && 0 && 1 && 1 && 0
\end{bmatrix}
(H^{(5)})^{-1}\\
&\times
\begin{bmatrix}
- \sum_{i=1}^M (\frac{1}{3!}(p_i \frac{\partial \psi_c}{\partial x}
+q_i
\frac{\partial \psi_c}{\partial y}
)^3
+\frac{1}{4!}(p_i \frac{\partial \psi_c}{\partial x}
+q_i
\frac{\partial \psi_c}{\partial y})^4+\cdots)p_i w_i^2\\
- \sum_{i=1}^M (\frac{1}{3!}(p_i \frac{\partial \psi_c}{\partial x}
+q_i
\frac{\partial \psi_c}{\partial y}
)^3
+\frac{1}{4!}(p_i \frac{\partial \psi_c}{\partial x}
+q_i
\frac{\partial \psi_c}{\partial y})^4+\cdots)
q_i w_i^2\\
- \sum_{i=1}^M (\frac{1}{3!}(p_i \frac{\partial \psi_c}{\partial x}
+q_i
\frac{\partial \psi_c}{\partial y}
)^3
+\frac{1}{4!}(p_i \frac{\partial \psi_c}{\partial x}
+q_i
\frac{\partial \psi_c}{\partial y})^4+\cdots)
\frac{p_i^2}{2} w_i^2\\
- \sum_{i=1}^M (\frac{1}{3!}(p_i \frac{\partial \psi_c}{\partial x}
+q_i
\frac{\partial \psi_c}{\partial y}
)^3
+\frac{1}{4!}(p_i \frac{\partial \psi_c}{\partial x}
+q_i
\frac{\partial \psi_c}{\partial y})^4+\cdots)
\frac{q_i^2}{2} w_i^2\\
- \sum_{i=1}^M (\frac{1}{3!}(p_i \frac{\partial \psi_c}{\partial x}
+q_i
\frac{\partial \psi_c}{\partial y}
)^3
+\frac{1}{4!}(p_i \frac{\partial \psi_c}{\partial x}
+q_i
\frac{\partial \psi_c}{\partial y})^4+\cdots)
p_i q_i w_i^2
\end{bmatrix}\\
&=
-
\sum_{i=1}^M  
\frac{1}{3!}(p_i \frac{\partial \psi_c}{\partial x}
+q_i
\frac{\partial \psi_c}{\partial y}
)^3
\theta_i
-
\sum_{i=1}^M  \frac{1}{4!}(p_i \frac{\partial \psi_c}{\partial x}
+q_i
\frac{\partial \psi_c}{\partial y})^4 \theta_i -\cdots.
\end{split}
\end{equation*}
Therefore, as $\delta\to 0$ and $\Delta t\to 0$, the
first term in \eqref{eq:eq12} tends to zero.

\subsection{Proof for Three-dimensional Case}
The proof follows the same steps as in proof \ref{sec:appendix2}.
However, when we want to prove that the first term in \eqref{eq:eq12} tends to zero as $\delta \to 0, \Delta t\to 0$, we proceed as follows
\begin{equation}
\label{eq:eq13}
\begin{split}
 \nabla^2 \psi (\mathbf{x}_c ,k\Delta t)&-(-\theta_c \psi (\mathbf{x}_c,k\Delta t)+\sum_{i=1} ^M \theta_i \psi (\mathbf{x}_i,k\Delta t)) \\
 &= 
 \begin{bmatrix}
0 & 0 & 0 & 1 & 1 &1 & 0 & 0 & 0
\end{bmatrix}
(\mathbf{d}^{(9)}-(H^{(9)})^{-1}\mathbf{f}^{(9)}).
\end{split}
\end{equation}
Suppose  that we use higher order derivatives in the Taylor series expansion. By  minimizing the norm $A^{(9)}(\psi)$ with respect to  the space variables, we  get  
\begin{equation*}
 H^{(9)}  \mathbf{d}^{(9)}=  \tilde{\mathbf{f}}^{(9)}.
\end{equation*}
By substituting $\mathbf{d}^{(9)}$ in
\eqref{eq:eq13}, we have
{\small
\begin{equation*}
\begin{split}
 \nabla^2 & \psi (\mathbf{x}_c ,k\Delta t)-(-\theta_c \psi (\mathbf{x}_c,k\Delta t)+\sum_{i=1} ^M \theta_i \psi (\mathbf{x}_i,k\Delta t)) \\
 &= 
 \begin{bmatrix}
0 & 0 & 0 & 1 & 1 &1 & 0 & 0 & 0
\end{bmatrix}
(H^{(9)})^{-1}(\tilde{\mathbf{f}}^{(9)}-\mathbf{f}^{(9)})\\
&= 
 \begin{bmatrix}
0 & 0 & 0 & 1 & 1 &1 & 0 & 0 & 0
\end{bmatrix}
(H^{(9)})^{-1}\\
&\times
\begin{bmatrix}
- \sum_{i=1}^M (\frac{1}{3!}(p_i \frac{\partial \psi_c}{\partial x}
+q_i
\frac{\partial \psi_c}{\partial y}
+r_i
\frac{\partial \psi_c}{\partial z})^3
+\frac{1}{4!}(p_i \frac{\partial \psi_c}{\partial x}
+q_i
\frac{\partial \psi_c}{\partial y}
+r_i
\frac{\partial \psi_c}{\partial z}
)^4+\cdots)p_i w_i^2\\
- \sum_{i=1}^M (\frac{1}{3!}(p_i \frac{\partial \psi_c}{\partial x}
+q_i
\frac{\partial \psi_c}{\partial y}
+r_i
\frac{\partial \psi_c}{\partial z})^3
+\frac{1}{4!}(p_i \frac{\partial \psi_c}{\partial x}
+q_i
\frac{\partial \psi_c}{\partial y}
+r_i
\frac{\partial \psi_c}{\partial z}
)^4+\cdots)q_i w_i^2\\
- \sum_{i=1}^M (\frac{1}{3!}(p_i \frac{\partial \psi_c}{\partial x}
+q_i
\frac{\partial \psi_c}{\partial y}
+r_i
\frac{\partial \psi_c}{\partial z})^3
+\frac{1}{4!}(p_i \frac{\partial \psi_c}{\partial x}
+q_i
\frac{\partial \psi_c}{\partial y}
+r_i
\frac{\partial \psi_c}{\partial z}
)^4+\cdots)r_i w_i^2\\
- \sum_{i=1}^M (\frac{1}{3!}(p_i \frac{\partial \psi_c}{\partial x}
+q_i
\frac{\partial \psi_c}{\partial y}
+r_i
\frac{\partial \psi_c}{\partial z})^3
+\frac{1}{4!}(p_i \frac{\partial \psi_c}{\partial x}
+q_i
\frac{\partial \psi_c}{\partial y}
+r_i
\frac{\partial \psi_c}{\partial z}
)^4+\cdots)
\frac{p_i^2}{2} w_i^2\\
- \sum_{i=1}^M (\frac{1}{3!}(p_i \frac{\partial \psi_c}{\partial x}
+q_i
\frac{\partial \psi_c}{\partial y}
+r_i
\frac{\partial \psi_c}{\partial z})^3
+\frac{1}{4!}(p_i \frac{\partial \psi_c}{\partial x}
+q_i
\frac{\partial \psi_c}{\partial y}
+r_i
\frac{\partial \psi_c}{\partial z}
)^4+\cdots)
\frac{q_i^2}{2} w_i^2\\
- \sum_{i=1}^M (\frac{1}{3!}(p_i \frac{\partial \psi_c}{\partial x}
+q_i
\frac{\partial \psi_c}{\partial y}
+r_i
\frac{\partial \psi_c}{\partial z})^3
+\frac{1}{4!}(p_i \frac{\partial \psi_c}{\partial x}
+q_i
\frac{\partial \psi_c}{\partial y}
+r_i
\frac{\partial \psi_c}{\partial z}
)^4+\cdots)
\frac{r_i^2}{2} w_i^2\\
- \sum_{i=1}^M (\frac{1}{3!}(p_i \frac{\partial \psi_c}{\partial x}
+q_i
\frac{\partial \psi_c}{\partial y}
+r_i
\frac{\partial \psi_c}{\partial z})^3
+\frac{1}{4!}(p_i \frac{\partial \psi_c}{\partial x}
+q_i
\frac{\partial \psi_c}{\partial y}
+r_i
\frac{\partial \psi_c}{\partial z}
)^4+\cdots)
p_i q_i w_i^2\\
- \sum_{i=1}^M (\frac{1}{3!}(p_i \frac{\partial \psi_c}{\partial x}
+q_i
\frac{\partial \psi_c}{\partial y}
+r_i
\frac{\partial \psi_c}{\partial z})^3
+\frac{1}{4!}(p_i \frac{\partial \psi_c}{\partial x}
+q_i
\frac{\partial \psi_c}{\partial y}
+r_i
\frac{\partial \psi_c}{\partial z}
)^4+\cdots)
p_i r_i w_i^2\\
- \sum_{i=1}^M (\frac{1}{3!}(p_i \frac{\partial \psi_c}{\partial x}
+q_i
\frac{\partial \psi_c}{\partial y}
+r_i
\frac{\partial \psi_c}{\partial z})^3
+\frac{1}{4!}(p_i \frac{\partial \psi_c}{\partial x}
+q_i
\frac{\partial \psi_c}{\partial y}
+r_i
\frac{\partial \psi_c}{\partial z}
)^4+\cdots)
q_i r_i w_i^2
\end{bmatrix}\\
&=
-
\sum_{i=1}^M  
\frac{1}{3!}(p_i \frac{\partial \psi_c}{\partial x}
+q_i
\frac{\partial \psi_c}{\partial y}
+r_i
\frac{\partial \psi_c}{\partial z}
)^3
\theta_i
-
\sum_{i=1}^M  \frac{1}{4!}(p_i \frac{\partial \psi_c}{\partial x}
+q_i
\frac{\partial \psi_c}{\partial y}
+r_i
\frac{\partial \psi_c}{\partial z}
)^4 \theta_i -\cdots.
\end{split}
\end{equation*}
}
We can conclude that as $\delta\to 0$ and $\Delta t\to 0$, the
first term in \eqref{eq:eq12} tends to zero. This concludes the proof.
\end{appendices}

\bibliography{sn-bibliography}

\end{document}